\documentclass[11pt]{article}

\usepackage[english]{babel}
\usepackage{amsmath,amsthm,amssymb,bm,thmtools,enumitem,mathtools,titlesec,tikz}
\usepackage[left=2cm,right=2cm,top=2cm,bottom=2cm,bindingoffset=0cm]{geometry}

\setlist[enumerate]{topsep=0pt,label=\textup{(\roman*)},leftmargin=\parindent,labelsep=.5em}
\setlist{noitemsep}

\usepackage{quoting}
\quotingsetup{vskip=0pt}

\usepackage[T1]{fontenc}

\declaretheoremstyle[
  spaceabove=\topsep, spacebelow=6pt,
  headfont=\normalfont\bfseries,
  notefont=\mdseries, notebraces={(}{)},
  bodyfont=\normalfont\itshape,
  postheadspace=.5em,
  qed=\qedsymbol
]{mystyle}

\declaretheoremstyle[
  spaceabove=\topsep, spacebelow=6pt,
  headfont=\normalfont\bfseries,
  notefont=\mdseries, notebraces={(}{)},
  bodyfont=\normalfont,
  postheadspace=.5em,
  qed=\qedsymbol
]{mydefstyle}

\theoremstyle{mystyle}
\declaretheorem[numberlike=subsection]{proposition}

\declaretheorem[numbered=no,name=Theorem]{thm*}

\declaretheorem[numbered=no,name=Corollary]{cor*}
\declaretheorem[numberlike=subsection]{lemma}

\declaretheorem[numbered=no,name=Coleman--Oort Conjecture]{COConj}
\declaretheorem[numbered=no,name=Main Theorem]{MainA}

\theoremstyle{mydefstyle}
\declaretheorem[numberlike=subsection]{definition}

\declaretheorem[numberlike=subsection]{remarks}

\numberwithin{equation}{subsection}

\titleformat{\section}[block]
  {\filcenter\normalfont\large\bfseries}{\thesection.}{.5em}{}
\titleformat{\subsection}[runin]
  {\normalfont\bfseries}{\thesubsection.}{0em}{}

\titlespacing*{\section} {0pt}{6ex plus .2ex minus 1ex}{3ex plus .2ex minus 1ex}
\titlespacing*{\subsection} {0pt}{2ex plus .2ex minus 1ex}{.5em}

\newcommand{\mA}{\mathbb{A}}
\newcommand{\mC}{\mathbb{C}}

\newcommand{\mH}{\mathbb{H}}
\newcommand{\mP}{\mathbb{P}}
\newcommand{\mQ}{\mathbb{Q}}
\newcommand{\mR}{\mathbb{R}}
\newcommand{\mS}{\mathbb{S}}

\newcommand{\Af}{\mA_{\mathrm{f}}}

\newcommand{\Gal}{\mathrm{Gal}}

\newcommand{\GU}{\mathrm{GU}}
\newcommand{\Hom}{\mathrm{Hom}}

\newcommand{\Jac}{\mathrm{Jac}}
\newcommand{\Lie}{\mathrm{Lie}}
\newcommand{\PGL}{\mathrm{PGL}}
\newcommand{\PGU}{\mathrm{PGU}}
\newcommand{\Res}{\mathrm{Res}}
\newcommand{\Sh}{\mathrm{Sh}}
\newcommand{\SO}{\mathrm{SO}}
\newcommand{\Stab}{\mathrm{Stab}}

\newcommand{\ad}{\mathrm{ad}}
\newcommand{\comp}{\mathrm{c}}
\newcommand{\dec}{\mathrm{dec}}
\newcommand{\diag}{\mathrm{diag}}

\newcommand{\nc}{\mathrm{nc}}
\newcommand{\rank}{\mathrm{rk}}

\newcommand{\isomarrow}{\xrightarrow{~\sim~}}

\begin{document}

\begin{center}
\textbf{\Large The Coleman--Oort conjecture: reduction to three key cases}\footnote{The main theorem of this paper is based on notes written in 2011, which were long left unfinished because I was (and still am) unable to push this result further. My interest in this was rekindled by an invitation of Rachel Pries and Andrew Sutherland to talk about the Coleman--Oort conjecture in the VaNTAGe seminar, for which I should like to thank them.}
\bigskip

\textit{by}
\bigskip

{\Large Ben Moonen}
\end{center}
\vspace{8mm}

{\small 

\noindent
\begin{quoting}
\textbf{Abstract.} We show that the Coleman--Oort conjecture can be reduced to three particular cases. As an application we extend a result of Lu and Zuo, to the effect that for $g \geq 8$ the Coleman--Oort conjecture is true on the hyperelliptic locus.
\medskip

\end{quoting}

} 
\vspace{4mm}

\section*{Introduction}

It has been conjectured by Coleman (\cite{Coleman}, Conjecture~6) that for $g$ sufficiently large, there are only finitely many complete nonsingular complex curves~$C$ of genus~$g$, up to isomorphism, such that $\Jac(C)$ is a CM abelian variety. By the Andr\'e--Oort Conjecture, now a theorem of Tsimerman~\cite{Tsimerman} (building on contributions of many others), Coleman's conjecture is equivalent to the following one.

\begin{COConj}
Let $\mathsf{A}_g$ be the coarse moduli space of $g$-dimensional principally polarized complex abelian varieties. If $g$ is sufficiently large, there are no special subvarieties $S \subset \mathsf{A}_g$ of positive dimension that are contained in the Torelli locus $\mathsf{T}_g \subset \mathsf{A}_g$ and that meet the open Torelli locus~$\mathsf{T}_g^\circ$.
\end{COConj}

In this conjecture we should take $g$ to be at least~$8$, as for all smaller genera there are counterexamples. We refer to~\cite{BMFO} for further discussion.

Over the last years, the Coleman--Oort conjecture has attracted a lot of attention, and many mathematicians have contributed to a better understanding of it. Our modest goal here is to point out that the conjecture may be reduced to a couple of key cases. This is inspired by the following result, which combines Hain's results in~\cite{Hain} with some refinements due to de Jong and Zhang~\cite{dJZ}.

\begin{thm*}[Hain, de Jong--Zhang]
Let $S \subset \mathsf{A}_g$ with $g \geq 4$ be a special subvariety of positive dimension such that $S \subset \mathsf{T}_g$ and $S$ meets~$\mathsf{T}_g^\circ$. Assume the generic Mumford--Tate group on~$S$ has a $\mQ$-simple adjoint group. Then at least one of the following is true:
\begin{enumerate}[left= 0pt .. 3em,label=\textup{(\textup{H}\arabic*)}]
\item\label{H1} $S$ is a ball quotient,

\item\label{H2} all irreducible components of the intersection of~$S$ with $\mathsf{T}_g^\dec = \mathsf{T}_g\setminus \mathsf{T}_g^\circ$ have codimension~$\leq 2$ in~$S$,

\item\label{H3} the Baily--Borel compactification $S \subset S^*$ has a non-empty boundary and at least one irreducible component of the boundary has codimension~$\leq 2$ in~$S^*$.
\end{enumerate}
\end{thm*}

Interesting further results were obtained by Andreatta~\cite{Andreatta}.

De Jong and Zhang in their paper remark that it is complicated to list all Shimura varieties satisfying one of these conditions. However, this problem is greatly simplified if one observes that for the Coleman--Oort conjecture it suffices to consider special subvarieties $S \subset \mathsf{A}_g$ that are \emph{minimal}, in the sense that they do not contain smaller special subvarieties of positive dimension. Analyzing conditions \ref{H1}--\ref{H3} under this minimality assumption gives the following result.

\begin{MainA}
Let $g$ be an integer with $g \geq 8$. Assume there are infinitely many nonsingular complex curves of genus~$g$ whose Jacobians are CM abelian varieties. Then there exists a special subvariety $S \subset \mathsf{T}_g$ of positive dimension such that $S \cap \mathsf{T}_g^\circ \neq \emptyset$, and such that at least one of the following is true:
\begin{enumerate}[left= 0pt .. 2em,label=\textup{(\arabic*)}]
\item\label{dim1} $\dim(S) = 1$,

\item\label{dim2} $\dim(S) = 2$ and $S$ is complete,

\item\label{completeBQ} $S$ is a complete ball quotient.
\end{enumerate}
\end{MainA}

If $S$ is given by a Shimura datum $(G,X)$ then $S$ is complete if and only if $\rank_\mQ(G^\ad) = 0$.
\smallskip

An easy corollary is the following sharpening of the main result of the paper~\cite{LuZuo} by Lu and~Zuo.

\begin{cor*}
For $g \geq 8$, there are finitely many nonsingular complex hyperelliptic curves~$C$ of genus~$g$ for which $\Jac(C)$ is a CM abelian variety.
\end{cor*}

In the final section we show there exist minimal complete ball quotients of arbitrarily high dimension.


\section{Minimal special subvarieties}
\label{sec:minimal}

\subsection{}
\label{ssec:ShimData}
By a Shimura datum we mean a pair $(G,X)$ consisting of a reductive $\mQ$-group~$G$ and a $G(\mR)$-conjugacy class of homomorphisms $\mS \to G_\mR$ that satisfies conditions (2.1.1.1--3) in Section~2.1 of Deligne's paper~\cite{DelCorvalis}. We always assume that $G$ is the generic Mumford--Tate group on~$X$, i.e., there is no $\mQ$-subgroup $G^\prime \subsetneq G$ such that all $h \in X$ factor through~$G^\prime_\mR$.

For $h \in X$, let $M_h \subset G$ be its Mumford-Tate group, i.e., the smallest algebraic subgroup $M \subset G$ (over~$\mQ$) such that $h$ factors through~$M_\mR$. Let $Y \subset X$ be the  $M_h(\mR)$-orbit of~$h$. The pair $(M_h,Y)$ is again a Shimura datum, and $Y$ is a finite union of totally geodesic submanifolds of~$X$. The submanifolds of~$X$ that arise in this manner are called the \emph{Hodge loci} in~$X$.

Let $\Sh_K(G,X)$ be the complex Shimura variety $G(\mQ)\backslash X \times G(\Af)/K$. An irreducible algebraic subvariety $S \subset \Sh_K(G,X)$ is called a \emph{special subvariety} if there exists a Hodge locus $Y^+ \subset X$ and a class $\gamma K$ in $G(\Af)/K$ such that $S(\mC)$ is the image of $Y^+ \times \{\gamma K\} \subset X \times G(\Af)/K$ in $\Sh_K(G,X)$.

If $S \subset \Sh_K(G,X)$ is a special subvariety and $[h,\gamma K]$ is a Hodge-generic point in~$S$, then up to isomorphism the associated Shimura datum $(M_h,Y)$ only depends on~$S$. This justifies calling $(M_h,Y)$ \emph{the} Shimura datum corresponding to~$S$.

\begin{definition}
\label{def:minimal}
\begin{enumerate}
\item Let $S$ be a special subvariety of a Shimura variety $\Sh_K(G,X)$. We say that $S$ is \emph{minimal} if $\dim(S) > 0$ and $S$ contains no special subvarieties $S^\prime \subsetneq S$ of positive dimension.

\item We say that a Shimura datum $(G,X)$ is \emph{minimal} if $\dim_\mC(X) > 0$, and if for every $h \in X$ its $M_h(\mR)$-orbit is either a point or all of~$X$.
\end{enumerate}
\end{definition}

\begin{remarks}
\begin{enumerate}
\item If $S \subset \Sh_K(G,X)$ is a special subvariety and $(M,Y)$ is the corresponding Shimura datum, $S$ is minimal if and only if $(M,Y)$ is.

\item A Shimura datum $(G,X)$ is minimal if and only if the corresponding adjoint datum $(G^\ad,X^\ad)$ is minimal.

\end{enumerate}
\end{remarks}

\begin{lemma}
\label{lem:Gadsimple}
If an adjoint Shimura datum $(G^\ad,X^\ad)$ is minimal, $G^\ad$ is a $\mQ$-simple group.
\end{lemma}

\begin{proof} Suppose not. Then there is a decomposition of Shimura data $(G^\ad,X^\ad) = (H_1,Y_1) \times (H_2,Y_2)$ with $Y_1$ and~$Y_2$ of positive dimension. Choose a special point $y_2 \in Y_2$ and a Hodge-generic $y_1 \in Y_1$. The $M_h(\mR)$-orbit of~$(y_1,y_2)$ is $Y_1 \times \{y_2\}$, but this contradicts the minimality of $(G,X)$.
\end{proof}

\subsection{}
\label{subsec:ResFQH}
Suppose $(G^\ad,X^\ad)$ is an adjoint Shimura datum such that $G^\ad$ is $\mQ$-simple. As explained in \cite{DelCorvalis}, Section~2.3.4, we then have $G^\ad = \Res_{E/\mQ}\; H$ where $E$ is a totally real number field and $H$ is an absolutely simple adjoint group over~$E$. The set~$\Sigma$ of real embeddings of~$E$ decomposes as $\Sigma = \Sigma_\comp \coprod \Sigma_\nc$ where $\Sigma_\comp$ (resp.\ $\Sigma_\nc$) is the set of embeddings $\sigma \colon E \to \mR$ for which $H_\sigma := H \otimes_{E,\sigma} \mR$ is compact (resp.\ noncompact). The $\mQ$-rank of~$G^\ad$ is equal to the $E$-rank of~$H$.

\begin{proposition}
\label{prop:MinBallQuot}
Let $S \subset \mathsf{A}_g$ be a minimal special subvariety such that either $\dim(S) = 2$, or $S$ is a ball quotient with $\dim(S) > 1$, or $S$ is a quotient of the $3$-dimensional hermitian symmetric space of type~$\mathrm{BD\, I}$. Then $S$ is complete.
\end{proposition}

The argument in fact works in many more cases but these are the only cases we need.

\begin{proof}
Suppose $S$ is not complete. Let $(G,X)$ be corresponding Shimura datum. By Lemma~\ref{lem:Gadsimple}, $G^\ad$ is $\mQ$-simple. With notation $G^\ad = \Res_{E/\mQ}\; H$ as above, the assumption that $S$ is not complete means that $\rank_\mQ(G^\ad) = \rank_E(H)$ is positive. Therefore, all factors $H_\sigma$ are non-compact.

First assume $S$ is a ball quotient. Then $G^\ad_\mR$ has only one noncompact factor, so we must have $E = \mQ$ and $G^\ad_\mR \cong \PGU(n,1)$ with $n = \dim(S)$. There then exist:
\begin{itemize}
\item a division algebra~$\Delta$ with an involution $x \mapsto \bar{x}$ of the second kind, such that the centre $K = Z(D)$ is an imaginary quadratic field,
\item a right $\Delta$-module~$V$ of finite type,
\item a hermitian form $\phi \colon V \times V \to \Delta$  with respect to the given involution,
\end{itemize}
such that $G^\ad \cong \PGU(V,\phi)$. Because $G^\ad$ has positive rank, the Witt index of~$\phi$ is positive. The rank of~$G^\ad_\mR$ is then at least the degree of~$\Delta$. But we know that $\rank(G^\ad_\mR) = 1$, so it follows that $\Delta = K$ is an imaginary quadratic field. The domain~$X^\ad$ can be identified with the space of all negative lines in the $\mC$-vector space $V \otimes_\mQ \mR$ with respect to~$\phi$. The assertion now follows from the remark that this Shimura datum $(G^\ad,X^\ad)$ is minimal only if $n=1$. Indeed, if $n>1$ we can choose an $n$-dimensional $K$-subspace $W \subsetneq V$ such that $\phi|_W$ is nondegenerate and again has Witt index~$1$. If $L \subset W \otimes_\mQ \mR$ is a very general negative line and $M_h$ is the Mumford--Tate group of the corresponding point $h \in X^\ad$, the $M_h(\mR)$-orbit of~$h$ has dimension~$n-1$, in contradiction with the minimality of~$S$.

The case where $S$ is a quotient of the $3$-dimensional space of type~$\mathrm{BD\, I}$ is similar. The assumption that $S$ is not complete implies that $E=\mQ$ and $G^\ad$ is absolutely imple. In this case, $G^\ad = \SO(V,\phi)$ where $V$ is a $5$-dimensional $\mQ$-vector space and $\phi$ is a symmetric bilinear form of Witt index~$>0$ (as otherwise $S$ would be complete), of signature $(1,4)$ or $(4,1)$. Further, $X^\ad$ is biholomorphic to the $3$-dimensional complex manifold of isotropic lines $\mC \cdot v \subset V_\mC$ such that $\phi(v,\bar{v}) \neq 0$. Choose a $3$-dimensional subspace $W \subset V$ such that $\phi$ is indefinite on~$W_\mR$. If $L \subset W_\mC$ is a very general isotropic line, the Mumford--Tate group~$M_h$ of the corresponding point $h \in X^\ad$ is isomorphic to $\SO(W,\phi|_W)$, and the $M_h(\mR)$-orbit of~$h$ is $1$-dimensional. By minimality of~$S$ this case is therefore excluded.

Finally, suppose $\dim(S) = 2$. As before, we have $G^\ad = \Res_{E/\mQ}\; H$ and if $S$ is not complete, all factors~$H_\sigma$ are noncompact. The case where $S$ is a ball quotient has already been dealt with. The only other possibility is that $[E:\mQ] = 2$ and that $H$ is of Lie type~$A_1$. Because $\rank_E(H) > 0$ we then have $H = \PGL_{2,E}$ and $(G^\ad,X^\ad)$ is (the adjoint datum of) the Hilbert modular Shimura datum in genus~$2$, which is well-known not to be minimal.
\end{proof}

\section{Special subvarieties with large Baily--Borel boundary}

\subsection{}
\label{ssec:PairsHY}
As in \cite{DelCorvalis}, section~1.2, we consider pairs $(H,Y)$ consisting of a simple adjoint group~$H$ over~$\mR$ and an $H(\mR)$-conjugacy class of non-trivial homomorphisms $\mS \to H$ such that the adjoint representation gives $\Lie(H)$ a Hodge structure of type $(-1,1) + (0,0) + (1,-1)$, and such that $\mathrm{Inn}\bigl(h(i)\bigr)$ is a Cartan involution.

As explained in loc.\ cit., $Y$ corresponds to a special vertex in the Dynkin diagram of~$H_\mC$. Below we list, for $H$ of Lie type A--D, the possibilities for the Satake diagram of~$H$ and the special vertex (circled). This merges the information contained in \cite{DelCorvalis}, Table~1.3.9, with the classification results in \cite{Helgason}, Chapter~10.

The proper boundary components of~$Y$ in its Baily-Borel compactification $Y \subset Y^*$ are given by the maximal $\mR$-parabolic subgroups $P \subset H$. The conjugacy class of~$P$ corresponds to a $\Gal(\mC/\mR)$-orbit~$I_P$ of white vertices in the Satake diagram~$\Delta$. If $I_P$ contains the special vertex, the boundary component $Y_P \subset Y^*$ corresponding to the conjugacy class of~$P$ is a finite set of points. Otherwise, consider the connected component of $\Delta\setminus I_P$ that contains the special vertex; the dimension of~$Y_P$ is then the dimension (as a complex manifold) of the corresponding hermitian symmetric space. We list those cases in which $Y_P$ has codimension at most~$2$ in~$Y^*$.
\bigskip\goodbreak

\noindent
\textbf{Type} $\bm{\mathrm{A}_\ell(p,q)}$; $\bm{\ell = p+q-1\geq 1}$. Special vertex: $p$ or~$q$; $\dim(Y)=pq$.

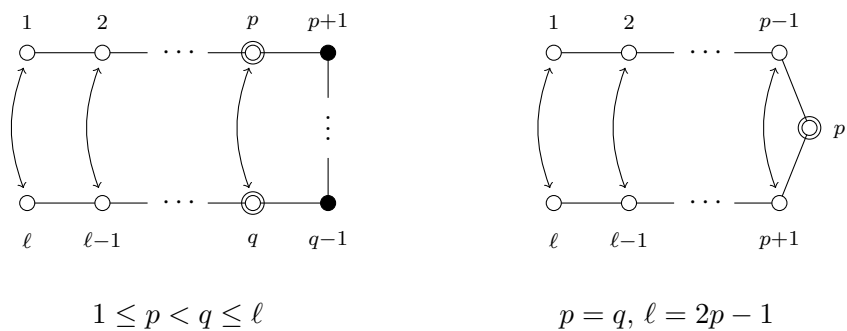
\begin{center}
\begin{tikzpicture}
\draw[fill=white] (3,0) circle (.15);
\draw[fill=white] (3,-2) circle (.15);

\draw (0,0) -- (1.6,0);
\draw (2.4,0) -- (4,0);
\draw (4,0) -- (4,-.6);
\draw (4,-1.4) -- (4,-2);
\draw (0,-2) -- (1.6,-2);
\draw (2.4,-2) -- (4,-2);

\draw[fill=white] (0,0) circle (.1);
\draw[fill=white] (1,0) circle (.1);
\draw[fill=white] (3,0) circle (.1);
\draw[fill=black] (4,0) circle (.1);
\draw[fill=white] (0,-2) circle (.1);
\draw[fill=white] (1,-2) circle (.1);
\draw[fill=white] (3,-2) circle (.1);
\draw[fill=black] (4,-2) circle (.1);

\node at (2.05,0) {$\cdots$};
\node at (2.05,-2) {$\cdots$};
\node at (4,-.9) {$\vdots$};

\node at (0,.4) {$\scriptstyle 1$};
\node at (1,.4) {$\scriptstyle 2$};
\node at (3,.4) {$\scriptstyle p$};
\node at (4,.4) {$\scriptstyle p+1$};
\node at (0,-2.5) {$\scriptstyle \ell$};
\node at (1,-2.5) {$\scriptstyle \ell-1$};
\node at (3,-2.5) {$\scriptstyle q$};
\node at (4,-2.5) {$\scriptstyle q-1$};

\draw[<->,in=110,out=-110] (-.05,-.2) to (-.05,-1.8);
\draw[<->,in=110,out=-110] (0.95,-.2) to (0.95,-1.8);
\draw[<->,in=110,out=-110] (2.95,-.2) to (2.95,-1.8);

\node at (2,-3.5) {$1 \leq p < q \leq \ell$};

\draw[fill=white] (10.4,-1) circle (.15);

\draw (7,0) -- (8.6,0);
\draw (9.4,0) -- (10,0);
\draw (10,0) -- (10.4,-1);
\draw (10.4,-1) -- (10,-2);
\draw (7,-2) -- (8.6,-2);
\draw (9.4,-2) -- (10,-2);

\draw[fill=white] (7,0) circle (.1);
\draw[fill=white] (8,0) circle (.1);
\draw[fill=white] (10,0) circle (.1);
\draw[fill=white] (10.4,-1) circle (.1);
\draw[fill=white] (7,-2) circle (.1);
\draw[fill=white] (8,-2) circle (.1);
\draw[fill=white] (10,-2) circle (.1);

\node at (9.05,0) {$\cdots$};
\node at (9.05,-2) {$\cdots$};

\node at (7,.4) {$\scriptstyle 1$};
\node at (8,.4) {$\scriptstyle 2$};
\node at (10,.4) {$\scriptstyle p-1$};
\node at (10.8,-1.05) {$\scriptstyle p$};
\node at (7,-2.5) {$\scriptstyle \ell$};
\node at (8,-2.5) {$\scriptstyle \ell-1$};
\node at (10,-2.5) {$\scriptstyle p+1$};

\draw[<->,in=110,out=-110] (6.95,-.2) to (6.95,-1.8);
\draw[<->,in=110,out=-110] (7.95,-.2) to (7.95,-1.8);
\draw[<->,in=110,out=-110] (9.95,-.2) to (9.95,-1.8);

\node at (8.5,-3.5) {$p=q$, $\ell=2p-1$};
\end{tikzpicture}
\end{center}

\noindent
The $\Gal(\mC/\mR)$-orbit $\{i,\ell+1-i\}$ for $1\leq i \leq p$ gives a boundary component of type $\mathrm{A}_{\ell-2i}(p-i,q-i)$. For $p=q$ we additionally have the Galois-orbit $\{p\}$, which gives a $0$-dimensional boundary component.

Boundary components of codimension~$1$: only for $p=q=1$, in which case $\dim(Y)=1$ and the boundary component is $0$-dimensional.

Boundary components of codimension~$2$: only for $p=1$ and $q=2$, in which case $\dim(Y)=2$ and we have a $0$-dimensional boundary component.

\bigskip\goodbreak

\noindent
\textbf{Type} $\bm{\mathrm{B}_\ell}$; $\bm{\ell \geq 2}$. Special vertex: $1$; trivial $\Gal(\mC/\mR)$-action; $\dim(Y)=2\ell-1$.

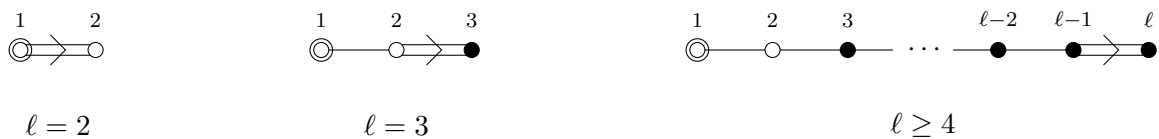
\begin{center}
\begin{tikzpicture}

\draw[fill=white] (0,0) circle (.15);
\draw (0,-.07) -- (1,-.07);
\draw (0,.07) -- (1,.07);
\draw (.4,0.2) -- (.6,0) -- (.4,-0.2);

\draw[fill=white] (0,0) circle (.1);
\draw[fill=white] (1,0) circle (.1);

\node at (0,.4) {$\scriptstyle 1$};
\node at (1,.4) {$\scriptstyle 2$};
\node at (.5,-1) {$\ell=2$};

\draw[fill=white] (4,0) circle (.15);
\draw (4,0) -- (5,0);
\draw (5,-.07) -- (6,-.07);
\draw (5,.07) -- (6,.07);
\draw (5.4,0.2) -- (5.6,0) -- (5.4,-0.2);

\draw[fill=white] (4,0) circle (.1);
\draw[fill=white] (5,0) circle (.1);
\draw[fill=black] (6,0) circle (.1);

\node at (4,.4) {$\scriptstyle 1$};
\node at (5,.4) {$\scriptstyle 2$};
\node at (6,.4) {$\scriptstyle 3$};
\node at (5,-1) {$\ell=3$};

\draw[fill=white] (9,0) circle (.15);
\draw (9,0) -- (11.6,0);
\draw(12.4,0) -- (14,0);
\draw (14,-.07) -- (15,-.07);
\draw (14,.07) -- (15,.07);
\draw (14.4,0.2) -- (14.6,0) -- (14.4,-0.2);

\draw[fill=white] (9,0) circle (.1);
\draw[fill=white] (10,0) circle (.1);
\draw[fill=black] (11,0) circle (.1);
\draw[fill=black] (13,0) circle (.1);
\draw[fill=black] (14,0) circle (.1);
\draw[fill=black] (15,0) circle (.1);

\node at (12.05,0) {$\cdots$};

\node at (9,.4) {$\scriptstyle 1$};
\node at (10,.4) {$\scriptstyle 2$};
\node at (11,.4) {$\scriptstyle 3$};
\node at (13,.4) {$\scriptstyle \ell-2$};
\node at (14,.4) {$\scriptstyle \ell-1$};
\node at (15,.4) {$\scriptstyle \ell$};
\node at (12,-1) {$\ell\geq 4$};
\end{tikzpicture}
\end{center}

\noindent
Boundary components of codimension~$2$: only for $\ell=2$, in which case $\dim(Y)=3$ and we have a $1$-dimensional boundary component corresponding to the Galois-orbit $\{2\}$.
\bigskip\goodbreak

\noindent
\textbf{Type} $\bm{\mathrm{C}_\ell}$; $\bm{\ell \geq 3}$. Special vertex: $\ell$; trivial $\Gal(\mC/\mR)$-action; $\dim(Y)=\ell(\ell+1)/2$.

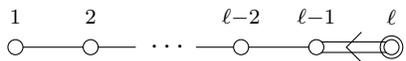
\begin{center}
\begin{tikzpicture}
\draw[fill=white] (5,0) circle (.15);
\draw (0,0) -- (1.6,0);
\draw(2.4,0) -- (4,0);
\draw (4,-.07) -- (5,-.07);
\draw (4,.07) -- (5,.07);
\draw (4.6,0.2) -- (4.4,0) -- (4.6,-0.2);

\draw[fill=white] (0,0) circle (.1);
\draw[fill=white] (1,0) circle (.1);
\draw[fill=white] (3,0) circle (.1);
\draw[fill=white] (4,0) circle (.1);
\draw[fill=white] (5,0) circle (.1);

\node at (2.05,0) {$\cdots$};

\node at (0,.4) {$\scriptstyle 1$};
\node at (1,.4) {$\scriptstyle 2$};
\node at (3,.4) {$\scriptstyle \ell-2$};
\node at (4,.4) {$\scriptstyle \ell-1$};
\node at (5,.4) {$\scriptstyle \ell$};
\end{tikzpicture}
\end{center}

\noindent
There are no cases that give a boundary component of codimension $\leq 2$.
\bigskip\goodbreak

\noindent
\textbf{Type} $\bm{\mathrm{D}_\ell^{\mR}}$; $\bm{\ell \geq 4}$. Special vertex: $1$;  trivial $\Gal(\mC/\mR)$-action; $\dim(Y)=2\ell-2$.

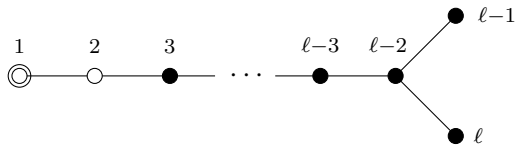
\begin{center}
\begin{tikzpicture}
\draw[fill=white] (9,0) circle (.15);
\draw (9,0) -- (11.6,0);
\draw(12.4,0) -- (14,0);
\draw (14,0) -- (14.8,.8);
\draw (14,0) -- (14.8,-.8);

\draw[fill=white] (9,0) circle (.1);
\draw[fill=white] (10,0) circle (.1);
\draw[fill=black] (11,0) circle (.1);
\draw[fill=black] (13,0) circle (.1);
\draw[fill=black] (14,0) circle (.1);
\draw[fill=black] (14.8,.8) circle (.1);
\draw[fill=black] (14.8,-.8) circle (.1);

\node at (12.05,0) {$\cdots$};

\node at (9,.4) {$\scriptstyle 1$};
\node at (10,.4) {$\scriptstyle 2$};
\node at (11,.4) {$\scriptstyle 3$};
\node at (13,.4) {$\scriptstyle \ell-3$};
\node at (13.9,.4) {$\scriptstyle \ell-2$};
\node at (15.35,.8) {$\scriptstyle \ell-1$};
\node at (15.1,-.8) {$\scriptstyle \ell$};
\end{tikzpicture}
\end{center}

\noindent
There are no cases that give a boundary component of codimension $\leq 2$.
\bigskip\goodbreak

\noindent
\textbf{Type} $\bm{\mathrm{D}_\ell^{\mH}}$; $\bm{\ell \geq 5}$. Special vertex: $\ell$;  $\dim(Y)=\ell(\ell-1)/2$.

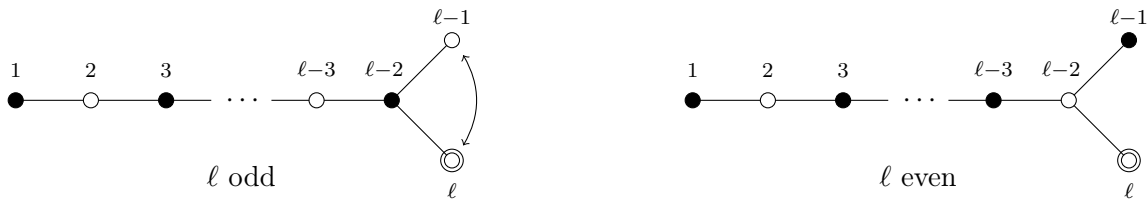
\begin{center}
\begin{tikzpicture}
\draw[fill=white] (5.8,-.8) circle (.15);
\draw (0,0) -- (2.6,0);
\draw(3.4,0) -- (5,0);
\draw (5,0) -- (5.8,.8);
\draw (5,0) -- (5.8,-.8);

\draw[fill=black] (0,0) circle (.1);
\draw[fill=white] (1,0) circle (.1);
\draw[fill=black] (2,0) circle (.1);
\draw[fill=white] (4,0) circle (.1);
\draw[fill=black] (5,0) circle (.1);
\draw[fill=white] (5.8,.8) circle (.1);
\draw[fill=white] (5.8,-.8) circle (.1);

\node at (3.05,0) {$\cdots$};

\node at (0,.4) {$\scriptstyle 1$};
\node at (1,.4) {$\scriptstyle 2$};
\node at (2,.4) {$\scriptstyle 3$};
\node at (4,.4) {$\scriptstyle \ell-3$};
\node at (4.9,.4) {$\scriptstyle \ell-2$};
\node at (5.8,1.1) {$\scriptstyle \ell-1$};
\node at (5.8,-1.2) {$\scriptstyle \ell$};
\node at (3,-1) {$\ell$ odd};

\draw[<->,in=60,out=-60] (5.95,.6) to (5.95,-.6);

\draw[fill=white] (14.8,-.8) circle (.15);
\draw (9,0) -- (11.6,0);
\draw(12.4,0) -- (14,0);
\draw (14,0) -- (14.8,.8);
\draw (14,0) -- (14.8,-.8);

\draw[fill=black] (9,0) circle (.1);
\draw[fill=white] (10,0) circle (.1);
\draw[fill=black] (11,0) circle (.1);
\draw[fill=black] (13,0) circle (.1);
\draw[fill=white] (14,0) circle (.1);
\draw[fill=black] (14.8,.8) circle (.1);
\draw[fill=white] (14.8,-.8) circle (.1);

\node at (12.05,0) {$\cdots$};

\node at (9,.4) {$\scriptstyle 1$};
\node at (10,.4) {$\scriptstyle 2$};
\node at (11,.4) {$\scriptstyle 3$};
\node at (13,.4) {$\scriptstyle \ell-3$};
\node at (13.9,.4) {$\scriptstyle \ell-2$};
\node at (14.8,1.1) {$\scriptstyle \ell-1$};
\node at (14.8,-1.2) {$\scriptstyle \ell$};
\node at (12,-1) {$\ell$ even};

\end{tikzpicture}
\end{center}

\noindent
There are no cases that give a boundary component of codimension $\leq 2$.

\subsection{}
Let $(G,X)$ be an adjoint Shimura datum of abelian type such that the group~$G$ is $\mQ$-simple. As in~\ref{subsec:ResFQH} we have $G = \Res_{E/\mQ}\, H$ for some totally real number field~$E$ and an absolutely simple adjoint group~$H$ over~$E$. The assumption that $(G,X)$ is of abelian type implies that $H$ is of Lie type A, B, C or~D. Let $\Sigma = \Sigma_\comp \coprod \Sigma_\nc$ be the set of real embeddings of~$E$. Then $X = \prod_{\sigma \in \Sigma_\nc}\, Y_\sigma$, where $Y_\sigma$ is an $H_\sigma(\mR)$-conjugacy class of homomorphisms $\mS \to H_\sigma$.

Suppose $\rank_\mQ(G) = \rank_E(H)$ is positive. All factors~$H_\sigma$ are then non-compact and $\dim(Y_\sigma) > 0$ for every $\sigma \in \Sigma$. Each pair $(H_\sigma,Y_\sigma)$ is of the type considered in~\ref{ssec:PairsHY}. Note that all $(H_\sigma,Y_\sigma)$ are of the same type $\mathrm{T}_\ell$ with $\mathrm{T} \in \{\mathrm{A},\mathrm{B},\mathrm{C},\mathrm{D}\}$ and $\ell \geq 1$, but that for type~$\mathrm{A}_\ell$ the parameters $(p,q)$ may depend on~$\sigma$. For $K \subset G(\Af)$ a compact open subgroup, the boundary components of~$\Sh_K(G,X)$ in the Baily-Borel compactification are given by the maximal $\mQ$-parabolic subgroups of~$G$. Any such is of the form $\Res_{E/\mQ}\, P$ with $P \subset H$ a maximal $E$-parabolic subgroup. For each $\sigma \in \Sigma$ the parabolic subgroup $P_\sigma \subset H_\sigma$ gives rise to a boundary component $Y_{P,\sigma} \subset Y_\sigma^*$, and the rational boundary component of the Baily-Borel compactification $X \subset X^*$ that corresponds to~$P$ is $\prod_{\sigma \in \Sigma}\, Y_{P,\sigma}$. Inspecting the tables in~\ref{ssec:PairsHY}, we find that the Baily-Borel compactification of~$X$ (equivalently: of $\Sh_K(G,X)$) can have a non-empty proper boundary component of codimension~$\leq 2$ only in the following cases:
\begin{enumerate}[left= 0pt .. 3em,label=\textup{(B\arabic*)}]
\item\label{caseA1} $[E:\mQ] \leq 2$ and $H = \PGL_{2,F}$;
\item\label{caseA2} $E=\mQ$ and $H=G$ is of type $\mathrm{A}_2(1,2)$;
\item\label{caseB2} $E=\mQ$ and $H=G$ is of type $\mathrm{B}_2$.
\end{enumerate}

\section{Proofs of the main results}

\subsection{}
\label{subsec:ST}
Let $\mathsf{T}_g \subset \mathsf{A}_g$ be the Torelli locus. Let $\mathsf{T}_g^\circ \subset \mathsf{T}_g$ be the open Torelli locus, i.e., the locus of Jacobians of smooth curves. The complement $\mathsf{T}_g^\dec \subset \mathsf{T}_g$ is the locus of decomposable Jacobians. We are interested in closed irreducible subvarieties $S \subset \mathsf{A}_g$ that satisfy the following condition:
\begin{description}[style=sameline,leftmargin=1.5em]
\item[(ST)] $S$ is a special subvariety of positive dimension, with $S \subset \mathsf{T}_g$ and $S \cap \mathsf{T}_g^\circ \neq \emptyset$
\end{description}
(ST for Special subvariety in the Torelli locus.)

\begin{lemma}
\label{lem:minimal}
Let $S \subset \mathsf{A}_g$ be a special subvariety that satisfies\/~\textup{(ST)} and such there is no $S^\prime \subsetneq S$ that again satisfies\/~\textup{(ST)}. Then $S$ is minimal.
\end{lemma}

\begin{proof}
If $S$ is not minimal, there exists a special subvariety $S^\prime \subsetneq S$ of positive dimension. Every Hecke translate of~$S^\prime$ is again a special subvariety of~$S$. These Hecke translates lie dense in~$S$, so there certainly exists one that is not fully contained in the boundary of~$\mathsf{T}_g$; but this contradicts the assumption.
\end{proof}

\subsection{} \emph{Proof of the Main Theorem.}
Suppose there exist infinitely many curves (complete, nonsingular) of genus~$g$ whose Jacobians are CM abelian varieties. By Tsimerman's theorem (formerly the Andr\'e--Oort conjecture), there then exists a special subvariety $S \subset \mathsf{A}_g$ that satisfies condition~(ST). By Lemma~\ref{lem:minimal} we may assume $S$ is minimal. As $G^\ad$ is then $\mQ$-simple, at least one of the conditions \ref{H1}--\ref{H3} in the theorem of Hain and de Jong--Zhang holds. (See the introduction.) We check what happens in these cases:

Case \ref{H1}: By Proposition~\ref{prop:MinBallQuot}, either $\dim(S) = 1$ or $S$ is a complete ball quotient, so either~\ref{dim1} or~\ref{completeBQ} in the theorem holds.

Case \ref{H2}: The irreducible components of $S \cap \mathsf{T}_g^\dec$ are again special subvarieties, so by minimality of~$S$ these are points; hence $\dim(S) \leq 2$. By Proposition~\ref{prop:MinBallQuot}, either~\ref{dim1} or~\ref{dim2} in the theorem holds.

Case \ref{H3}: This means that $S$ is not complete and we are in one of the cases \ref{caseA1}, \ref{caseA2} or~\ref{caseB2} at the end of the previous section. In case~\ref{caseA1} we have $\dim(S) \leq 2$, whereas in case~\ref{caseA2} $S$ is a ball quotient; these cases have already been covered. Case~\ref{caseB2}, finally, is excluded by Proposition~\ref{prop:MinBallQuot}. \qed

\subsection{} \emph{Proof of the Corollary.}
Suppose there are infinitely many hyperelliptic curves of genus~$g$ with a CM Jacobian. This gives us a special subvariety $S \subset \mathsf{A}_g$ satisfying~(ST) such that $S$ is contained in the closure of the hyperelliptic locus $\mathsf{HT}_g^\circ \subset \mathsf{T}_g^\circ$. We may assume~$S$ to be minimal, in which case it is of one of the three types \ref{dim1}--\ref{completeBQ} in the Main Theorem. First assume $S$ is complete. Because $\mathsf{HT}_g^\circ$ is affine, $S^\dec = S \cap \mathsf{T}_g^\dec$ has codimension~$1$ in~$S$. But $S^\dec$ is again a special subvariety, so by minimality we conclude that $\dim(S) = 1$. Now use that the case $\dim(S) = 1$ has been excluded by Lu and Zuo in their paper~\cite{LuZuo}. \qed

\section{The existence of minimal complete ball quotients}

As a complex ball (equivalently: a complex hyperbolic space) has totally geodesic complex submanifolds in every codimension, it is tempting to guess that ball quotients of dimension~$>1$ are not minimal. We here show that this is not true: there exist minimal complete ball quotients of arbitrarily large dimension.

\subsection{}
Let $p$ be a prime number. Let $K$ be an imaginary quadratic field, and let $\Delta$ be a division algebra of degree~$p$ over~$K$, equipped with an involution~$\tau$ of the second kind. We choose these data such that there exists an isomorphism $\alpha \colon \Delta \otimes_\mQ \mR \isomarrow M_p(\mC)$ via which $\tau$ corresponds to the adjoint involution of the hermitian form $x_1 \bar{y}_1 + \cdots + x_{p-1} \bar{y}_{p-1} - x_p \bar{y}_p$. Let $G = \GU(\Delta,\tau)$ be the group of unitary similitudes; see for instance \cite{BookInv}, \S 23. The isomorphism~$\alpha$ induces an isomorphism $G_\mR \cong \GU(p-1,1)$, via which we identify the two groups. Let $h \colon \mS \to G_\mR$ be the homomorphism that sends $z \in \mC^* = \mS(\mR)$ to the diagonal matrix $\diag(z,\ldots,z,\bar{z})$, and let $X \subset \Hom(\mS,G_\mR)$ be the $G(\mR)$-conjugacy class of~$h$. The pair $(G,X)$ is a Shimura datum of Hodge type. We identify~$X$ with the complex hyperbolic space $\mathbf{H}^{p-1}_\mC \subset \mP^{p-1}(\mC)$ through the map that sends~${}^gh$, for $g \in G(\mR)$, to the negative line in~$\mC^p$ spanned by the vector $g \cdot (0,\ldots,0,1)$. If $S \subset \mathsf{A}_g$ is a special subvariety with Shimura datum isomorphic to $(G,X)$ then $S$ is complete, as $\rank_\mQ(G^\ad) = 0$.

\begin{proposition}
The Shimura datum $(G,X)$ is minimal.
\end{proposition}

\begin{proof}
Suppose $(G,X)$ is not minimal, which means that there exists a sub-Shimura datum $(G^\prime,X^\prime) \subsetneq (G,X)$ with $\dim(X^\prime) > 0$. Let $X^{\prime,+}$ be a connected component of~$X^\prime$, which is a totally geodesic complex submanifold of~$X$. There is a unique linear subspace $H \subset \mC^p$ with $1 < \dim(H) < p$ such that $X^{\prime,+} = X \cap \mP(H)$; see \cite{Goldman}, Section~3.1.11. Let $G^\prime(\mR)^+$ be the (topological) identity component of~$G^\prime(\mR)$. The action of an element $g \in G^\prime(\mR)^+ \subset G(\mR)$ on~$X$ maps $X^{\prime,+} \subset X$ into itself.

The algebra $\Delta \otimes_\mQ \mR$ acts on~$\mC^p$. Let $\Stab(H) \subset (\Delta \otimes_\mQ \mR)$ be the stabilizer of the subspace~$H$. We have $G(\mR) \subset (\Delta \otimes_\mQ \mR)^*$ and this restricts to an inclusion $G^\prime(\mR)^+ \subset \Stab(H)^*$. Because $G^\prime$ is not abelian, we can choose two elements $g_1$, $g_2 \in G^\prime(\mQ)^+$ that do not commute. Let $K(g_1,g_2) \subset \Delta$ be the $K$-subalgebra generated by these elements. Because the degree of~$\Delta$ is prime we have $K(g_1,g_2) = \Delta$. On the other hand, $K(g_1,g_2) \subset \Delta \cap \Stab(H)$, where the intersection is take inside $\Delta \otimes_\mQ \mR$. It follows that $\Delta \otimes_\mQ \mR = \Stab(H)$, which is clearly impossible.
\end{proof}

\bigskip

\noindent
\texttt{b.moonen@science.ru.nl}

\noindent
Radboud University Nijmegen, IMAPP, Nijmegen, The Netherlands

\end{document}